\documentclass[amssymb, 11pt]{amsart}
\usepackage{latexsym}

\newlength{\standardunitlength}
\newtheorem{prop}{Proposition}[section]

\newtheorem{lemma}[prop]{Lemma}
\newtheorem{cor}[prop]{Corollary}
\newtheorem{theorem}[prop]{Theorem}

\newcommand{\AGL}{\operatorname{AGL}}
\newcommand{\GU}{\operatorname{GU}}
\newcommand{\AGU}{\operatorname{AGU}}
\newcommand{\ASU}{\operatorname{ASU}}
\newcommand{\SU}{\operatorname{SU}}
\newcommand{\GL}{\operatorname{GL}}
\newcommand{\SL}{\operatorname{SL}}
\newcommand{\ASL}{\operatorname{ASL}}
\newcommand{\Sp}{\operatorname{Sp}}
\newcommand{\ASp}{\operatorname{ASp}}
\newcommand{\SO}{\operatorname{SO}}

\newcommand{\OO}{\operatorname{O}}
\newcommand{\AO}{\operatorname{AO}}

\newcommand{\Z}{\mathbb{Z}}

\newcommand{\CC}{\mathbb{C}}

\begin{document}

\title{Enumeration of conjugacy classes in affine groups}

\author{Jason Fulman}
\address{Department of Mathematics, University of Southern California, Los Angeles, CA 90089-2532, USA }
\email{fulman@usc.edu}

\author{Robert M. Guralnick}
\address{Department of Mathematics, University of Southern California, Los Angeles, CA 90089-2532, USA}
\email{guralnic@usc.edu}

\date{December 30, 2021}

\thanks{
The first author was partially funded by a Simons Foundation Grant 400528.
The second author was partially supported by the NSF
grant DMS-1901595 and  Simons Foundation Fellowship 609771.}

\begin{abstract}
We study the conjugacy classes of the classical affine groups.  We derive generating functions
for the number of classes analogous to formulas of Wall and the authors for the classical groups.  We use these
to get good upper bounds for the number of classes.  These naturally come up as difficult cases in the study of the
non-coprime $k(GV)$ problem of Brauer.
\end{abstract}

\maketitle

\begin{center}
Dedicated to Pham Huu Tiep on the occasion of his 60th birthday
\end{center}

\section{Introduction}

Let $G$ be the group of affine transformations of a vector space $V$ over a finite field. In this paper we derive generating functions
for the number of conjugacy classes in this group and in the analogs for the other classical groups. For finite classical groups (not their affine versions), such generating functions were mostly obtained by
Wall \cite{W} (see also \cite{FG1} for orthogonal and symplectic groups in even characteristic).
  Besides the natural motivation for considering this, this is one of the most difficult cases
in the non-coprime $k(GV)$ problem introduced by Brauer to obtain results about characters.   This asks for bounds
on the number of conjugacy classes $k(H)$ where $H$ is a group with a normal abelian subgroup $V$.
One of the major results in this area based on work of many authors over a long period is that
$k(H) \le |V|$ if $\gcd(|H/V|,|V|)=1$. In fact there is an entire book devoted to this topic \cite{Sc}.  It turns out if we weaken this assumption, the result is no longer true but it still
is close.   One critical case is when $L = H/V$ acts irreducibly on $V$  (see \cite{GT} for reductions and for connections
with representation theory).   See \cite{GM, GT, K, Rob} for background and other results.   One would like to
prove that $k(H) < c|V|$ for some absolute constant $c$ (under suitable hypotheses). Another motivation for studying this is the relationship
with the conjugacy classes of the largest maximal parabolic subgroup of the classical groups.  See \cite{NaS} for the case of $\GL$.

In \cite{GT},  the focus was on the important case when $L$ is close to simple and the same bound was proved
in almost all cases studied.  One of the main cases left open was the case that $V$ is the natural module for
a classical group $L$.   It turns out that again aside from the case of $\AGL(n,q)$, the bound generally holds.
We show that $q^n \leq k(\AGL(n,q)) < (q^{n+1}  - 1)/(q-1) < 2q^n$ and obtain explicit and useful bounds in the analogs for other classical groups.

Variations on this theme and some other small families that were not considered in \cite{GT} will be studied in a sequel.

The paper is organized as follows.

Section \ref{gen} gives some preliminaries which are fundamental to our two approaches for calculating exact generating functions for $k(\AGL), k(\AGU)$ and $k(\ASp)$ and $k(\AO)$. The first approach writes $k(AG)$ as a weighted sum over conjugacy classes of $G$. We work this out for all cases except for the famously difficult cases of characteristic two symplectic and orthogonal groups. Our second approach enumerates irreducible representations instead of conjugacy classes. This allows us to calculate $k(AG)$ recursively, and has the additional benefit of working in both odd and even characteristic.

Section \ref{G} treats $k(\AGL(n,q))$, and also $k(AH)$ where $H$ is a group between $\GL(n,q)$ and $\SL(n,q)$. Section \ref{unita} treats $k(\AGU(n,q))$ and $k(AH)$ where $H$ is a group between $\GU(n,q)$ and $\SU(n,q)$. Section \ref{Symplec} treats the case $\ASp(2n,q)$. Section \ref{Orth} treats $\AO(n,q)$.

We dedicate this paper to Pham Huu Tiep, our friend and colleague, on the occasion of his 60th birthday. We note that he has done substantial work on the non coprime $k(GV)$ problem; see \cite{GT}.

\section{Preliminaries} \label{gen}

Let $G$ be a finite group and let $k$ be a finite field with $A$ a finite dimensional $kG$-module.   Then we consider
the group $H=AG$, the semidirect product of the normal subgroup $A$ and $G$.   We say that $H$ is the
corresponding affine group.   We will usually take
$A$ to be irreducible (and by replacing $k$ by $\mathrm{End}_G(A)$, we can assume that $A$ is absolutely irreducible).

Our first approach, which we call the $\bf{orbit}$ approach, expresses $k(AG)$ as a weighted sum over conjugacy classes of $G$. To describe this, let $[g,A]$ denote $(I-g)A$ where $I$ is the identity map. The
number of orbits of the centralizer $C_G(g)$ on $A/[g,A]$ depends only on the
conjugacy class $C$ of $g$, and we denote it by $o(C)$.  If $g$ and $x$ are elements of a group $G$, then we let $x^g=g^{-1}xg$.

\begin{lemma} \label{exact} Let $G$ and $A$ be as above. Then
\[ k(AG) = \sum_{C} o(C), \] where the sum is over all conjugacy
classes of $G$. \end{lemma}

\begin{proof}  Let $g \in C$ with $C$ a conjugacy class of $G$.   We need to show that the number of conjugacy
classes of elements $h \in AG$ such that $h$ is conjugate to some element of $gA$ is the number of orbits of $C_G(g)$ on $A/[g,A]$.

Suppose that $h=ga$.   Suppose that $gb$ is conjugate to $ga$.  Note that
$$
\{(ga)^b | b \in A\} = ga[g,A].
$$

Thus if $a, c \in A$,  $ga$ and $gc$ are conjugate in $H$ if and only if $a[g,A]$ and $b[g,A]$ are in
the same $C_G(g)$ orbit on $A/[g,A]$, whence the result.
\end{proof}

In this paper we find (for all cases except even characteristic symplectic and orthogonal groups) exact formulas for $o(C)$, which may be of independent interest. We then use these formulas, together with generating functions for $k(G)$, to find exact generating functions for $k(AG)$.

Our second approach, which we call the ${\bf character}$ approach, counts irreducible representations instead of conjugacy classes. This leads to recursive expressions for $k(AG)$. Together with known generating functions for $k(G)$, this enables us to obtain exact generating functions for $k(AG)$. One nice feature of the character approach is that it works in both odd and even characteristic.

Crucial to the character approach is the next lemma, which is a well known elementary exercise.

\begin{lemma} \label{l:semidirect} Let $G$ be a finite group and $V$ a finite $G$-module. Let $J=VG$ be the semidirect product. Let $\Delta$ be a set of $G$-orbit representatives on the set of irreducible characters of $V$. Then \[ k(J) = \sum_{\delta \in \Delta}  k(G_{\delta}) \] where $G_\delta$ is the stabilizer of the character $\delta$ in $J$.
\end{lemma}

\begin{proof}  Let $W$ be an irreducible $\CC J$-module.  Let $\delta$ be a character of $V$ that occurs in $W$ and set
$W_{\delta}$ to be the $\delta$ eigenspace of $V$.  Note that $\delta$ is unique  up to $G$-conjugacy and that the
stabilizer of $W_{\delta}$ in $J$ is precisely $VG_{\delta}$.  Thus,  $G_{\delta}$ acts irreducibly on $W_{\delta}$.
Conversely given any irreducible $G_{\delta}$-module $U$, we can extend it to a $J_{\delta}$ module by having
$V$ act via $\delta$.  Then inducing $U$ from $J_{\delta}$ gives an irreducible $J$-module.   Thus, we see
that $ k(J) = \sum_{\delta \in \Delta}  k(G_{\delta})$ as required.
\end{proof}

The following lemma is Euler's pentagonal number theorem (see for instance page 11 of \cite{A1}).

\begin{lemma} \label{pentag} For $q>1$,
\begin{eqnarray*}
\prod_{i \geq 1} \left( 1-\frac{1}{q^i} \right )
& = & 1 + \sum_{n=1}^{\infty} (-1)^n (q^{-n(3n-1)/2} + q^{-n(3n+1)/2} ) \\
& = & 1 - q^{-1} - q^{-2} + q^{-5} + q^{-7} - q^{-12} - q^{-15} + \cdots.
\end{eqnarray*}
\end{lemma}

A few times in this paper quantities which can be easily re-expressed in terms of the infinite product $\prod_{i=1}^{\infty} (1-1/q^i)$ will arise, and Lemma \ref{pentag} gives arbitrarily accurate upper and lower bounds on these products. Hence we will state bounds like
\[ \prod_{i=1}^{\infty} (1+\frac{1}{2^i}) = \prod_{i=1}^{\infty} \frac{(1-\frac{1}{4^i})}{(1-\frac{1}{2^i})} \leq 2.4 \] without explicitly mentioning Euler's pentagonal number theorem on each occasion.

We also use the following well-known lemma (see for instance \cite{O}).

\begin{lemma} \label{bit} Suppose that $f(u)$ is analytic for $|u|<R$. Let $M(r)$ denote the maximum of $|f|$ restricted to the circle $|u|=r$. Then for any $0<r<R$, the coefficient of $u^n$ in $f(u)$ has absolute value at most $M(r)/r^n$. \end{lemma}

As a final bit of notation, we let $|\lambda|$ denote the size of a partition $\lambda$.

\section{$\AGL$ and related groups} \label{G}

Section \ref{GLorb} uses the orbit approach to calculate the generating function for $k(\AGL(n,q))$. Section \ref{GLchar} uses the character approach to calculate the generating function for $k(\AGL(n,q))$ and related groups. Section \ref{GLbound} uses these generating functions to obtain bounds on $k(\AGL(n,q))$ and related groups.

\subsection{Orbit approach to $k(\AGL)$} \label{GLorb}

We use Lemma \ref{exact} to determine a generating function for the numbers
$k(\AGL(n,q))$.

The following lemma calculates $o(C)$ for a conjugacy class $C$ of $\GL(n,q)$.
This formula involves the number of distinct part sizes of a partition $\lambda$, which we denote by $d(\lambda)$. For example if $\lambda$ has 5 parts of size $4$,
3 parts of size $2$, and $4$ parts of size $1$, then $d(\lambda)=3$. If $\lambda$
is the empty partition, then $d(\lambda)=0$.

\begin{lemma} \label{oCalculateGL} Let $C$ be a conjugacy class of $\GL(n,q)$,
and let $\lambda_{z-1}(C)$ be the partition corresponding to the eigenvalue 1 in the rational canonical form of an element of $C$. Then
\[ o(C) = d(\lambda_{z-1}(C)) + 1 .\]
\end{lemma}

\begin{proof} Let $V$ be the natural module for $\GL(n,q)$.   Let $g \in C$.  Write $V = V_1 \oplus V_2$ where
$V_2= \ker(g-I)^n$.   Note that $[g,V]=V_1 \oplus V_2/[g,V_2]$ and the centralizer of $g$ preserves this decomposition.
Thus, we may assume that $V = V_2$, i.e. we may assume that $g$ is unipotent.

Now write $V = V_1 \oplus \ldots  \oplus V_m$ where $g|V_i$  has all Jordan blocks of size $i$.  We only consider the nonzero $V_i$.
So $ d_i = \dim V_i/[g,V_i]$ is the number of Jordan blocks of size $i$.  It is well known that the centralizer of $g$ induces the full
$\GL(d_i,q)$ and in particular any two nonzero elements of $V_i/[g,V_i]$  are in the same $C(g)$ orbit.

Consider $gv$ with $v = v_1 + \ldots + v_m$  with $v_i  \in V_i$.
    Note that if $h \in C(g)$,  then $hV_i \subset  V_1 \oplus \ldots \oplus V_i + [g,V]$. Thus, two elements in $V$ which are in the same $C(g)$-orbit module $[g,V]$ must have the same highest nonzero (modulo $[g,V]$) term.
Conversely, we need to show that any two such vectors are in the same orbit and indeed are in the orbit of $v_j$ with $v_j \in V_j \setminus [g,V_j]$.
By induction, we may assume that $j=m$.   Note that there exists $h \in C(g)$ so that $h$ is
  is trivial on $V/\sum_{e < m} V_e$ and $hv_m - v_m$ is an arbitrary element in $\oplus_{e < m} V_e/[g,V_e]$.
Thus, we see that the  $v$ and $v_m$ are in same the orbit.  Since $C(g)$ induces $\GL(d_m,q)$ on $V_m/[g,V_m]$ we see that orbit
representatives for $C(g)$ on $V[g,V]$ are $0$ and one vector $w_i \in V_i$ for each nonzero $V_i$.  The result follows.
\end{proof}

The following interesting identity will be helpful.

\begin{lemma} \label{distinct}
\[ \sum_{\lambda} \left[ d(\lambda)+1 \right] u^{|\lambda|}
= \frac{1}{1-u} \prod_{i \geq 1} \frac{1}{1-u^i} .\]
\end{lemma}

\begin{proof} Clearly
\[ \sum_{\lambda} q^{d(\lambda)} u^{|\lambda|} = \prod_{i \geq 1} \left( 1 + \frac{qu^i}{1-u^i}
\right) .\] Differentiate this equation with respect to $q$ and then set $q=1$. The left
hand side becomes \[ \sum_{\lambda} d(\lambda) u^{|\lambda|}. \] By the product rule, the
right hand side becomes
\begin{eqnarray*}
\sum_{i \geq 1} \frac{u^i}{1-u^i} \prod_{j \neq i} \left( 1 + \frac{u^j}{1-u^j} \right)
& = & \sum_{i \geq 1} \frac{u^i}{1-u^i} \prod_{j \neq i} \left( \frac{1}{1-u^j} \right) \\
& = & \left( \sum_{i \geq 1} u^i \right) \prod_{j \geq 1} \frac{1}{1-u^j}.
\end{eqnarray*}

Thus \begin{equation} \label{equal} \sum_{\lambda} d(\lambda) u^{|\lambda|} =
\left( \sum_{i \geq 1} u^i \right) \prod_{j \geq 1} \frac{1}{1-u^j}. \end{equation}

Since \[ \sum_{\lambda} u^{|\lambda|} = \prod_{j \geq 1} \frac{1}{1-u^j}, \] it follows from
(\ref{equal}) that
\begin{eqnarray*}
\sum_{\lambda} \left[ d(\lambda)+1 \right] u^{|\lambda|} & = & \left( \sum_{i \geq 0} u^i \right) \prod_{j \geq 1} \frac{1}{1-u^j} \\
& = & \frac{1}{1-u} \prod_{j \geq 1} \frac{1}{1-u^j},
\end{eqnarray*} as claimed.
\end{proof}

In what follows, for $d \geq 1$, we let $N(q;d)$ denote the number of monic irreducible polynomials $\phi(z)$ of degree
$d$ over $F_q$ for which $\phi(0) \neq 0$, that is monic irreducible polynomials other than $z$.

The following well known identity (see for example Theorem 3.25 of\cite{LN}) will be useful.

\begin{lemma} \label{prodpoly}
\[ \prod_{d \geq 1} (1-u^d)^{-N(q;d)} = \frac{1-u}{1-qu} .\]
\end{lemma}

Theorem \ref{AGL(n,q)} derives a generating function for the number of conjugacy classes
in $\AGL(n,q)$.

\begin{theorem} \label{AGL(n,q)}
\[ 1 + \sum_{n \geq 1} k(\AGL(n,q)) u^n = \frac{1}{1-u}
\prod_{i \geq 1} \frac{1-u^i}{1-qu^i} .\]
\end{theorem}

\begin{proof} By Lemma \ref{exact},
\[ 1 + \sum_{n \geq 1} k(\AGL(n,q)) u^n = 1 + \sum_{n \geq 1} u^n \sum_{C} o(C), \]
where the sum is over all conjugacy classes of $\GL(n,q)$.

Since conjugacy classes of $\GL(n,q)$ correspond to rational canonical forms, it follows from the
previous equation and Lemma \ref{oCalculateGL} that
\begin{eqnarray*}
& & 1 + \sum_{n \geq 1} k(\AGL(n,q)) u^n \\
& = & \left( \sum_{\lambda} \left[ d(\lambda)+1 \right] u^{|\lambda|} \right)
\left( \sum_{\lambda} u^{|\lambda|} \right)^{N(q;1)-1}
\prod_{d \geq 2} \left( \sum_{\lambda} u^{d |\lambda|} \right)^{N(q;d)} \\
& = & \left( \sum_{\lambda} \left[ d(\lambda)+1 \right] u^{|\lambda|} \right)
\prod_{i \geq 1} \left( \frac{1}{1-u^i} \right)^{N(q;1)-1}
\prod_{d \geq 2} \prod_{i \geq 1} \left( \frac{1}{1-u^{di}} \right)^{N(q,d)}.
\end{eqnarray*}

By Lemma \ref{distinct} this is equal to
\[ \frac{1}{1-u} \prod_{d \geq 1} \prod_{i \geq 1} \left( \frac{1}{1-u^{di}} \right)^{N(q;d)}
= \frac{1}{1-u} \prod_{i \geq 1} \prod_{d \geq 1} \left( \frac{1}{1-u^{di}} \right)^{N(q;d)} .\]
Applying Lemma \ref{prodpoly} this simplifies to
\[ \frac{1}{1-u} \prod_{i \geq 1} \frac{1-u^i}{1-qu^i}, \] as claimed.
\end{proof}

\subsection{Character approach to $k(\AGL)$ and related groups} \label{GLchar}

We apply Lemma \ref{l:semidirect}. Note that if $\delta$ is the trivial character, then $G_{\delta}=G$.
We recall the case of $G=\GL(n,q)$ with $V$ the natural module.  The group $J$ is usually denoted
as $\AGL(n,q)$ the affine general linear group.   Note that in this case, $G$ has precisely two orbits
on $\Delta$ with one of them being the trivial character.  Note that the stabilizer of a nontrivial linear character
is isomorphic to $\AGL(n-1,q)$
and so:

\begin{lemma} \label{l:agl}
$$k(\AGL(n,q)) = k(\GL(n,q)) + k(\AGL(n-1,q))= 1 + \sum_{m=1}^n k(\GL(m,q)).$$
\end{lemma}

As a corollary, we get another proof of Theorem \ref{AGL(n,q)}.

\begin{proof} Lemma \ref{l:agl} implies that
\[ 1 + \sum_{n \geq 1} k(\AGL(n,q)) u^n = \frac{1}{1-u} \left( 1 + \sum_{n \geq 1} k(\GL(n,q)) u^n \right).\] The result now follows from Macdonald's theorem \cite{M}
\[ 1 + \sum_{n \geq 1} k(\GL(n,q)) u^n = \prod_{i \geq 1} \frac{1-u^i}{1-qu^i}.\]
\end{proof}

\begin{lemma}  \label{bet} Fix q and let $n \ge 2$.  Let $\SL(n,q) \le H = H(n,q)
\le \GL(n,q)$
with $e=[H:\SL(n,q)]$.
\begin{enumerate}
\item $k(AH) = k(H) + k(AH(n-1,q))$.
\item  $k(AH) =   (q-1)/e     + \sum_{i=1}^n  k(H(i,q))$.
\end{enumerate}
\end{lemma}

\begin{proof}   The first statement follows exactly as in the proof of the case
of $\GL(n,q)$.    Note that $k(AH(1,q))= e + (q-1)/e = k(H(1,q)) + (q-1)/e$.

So iterating, we see that
$$
k(AH) = k(AH(1,q)) + \sum_{j=2}^n  k(H(j,q)) = (q-1)/e     + \sum_{i=1}^n  k(H(i,q)).
$$
\end{proof}

\subsection{Bounds on $k(\AGL)$ and related groups} \label{GLbound}

The following result is an interesting Corollary of Theorem \ref{AGL(n,q)}. If
$f(u)=\sum_{n \geq 0} f(n) u^n$ and $g(u)=\sum_{n \geq 0} g(n) u^n$, we use the
notation $f>>g$ to mean that $f(n) \geq g(n)$ for all $n \geq 0$.

\begin{cor}   $k(\AGL(1,q))  = q$ and for $n  \geq 2$,
\[ q^n < k(\AGL(n,q)) < 2 q^n.\]
\end{cor}

\begin{proof} By Theorem \ref{AGL(n,q)}, the fact that $q^n \leq k(\AGL(n,q))$
is equivalent to the statement
that \[ \frac{1}{1-u} \prod_{i \geq 1} \frac{1-u^i}{1-qu^i} >>
\frac{1}{1-uq}.\] Now notice that \[ \frac{1}{1-u} \prod_{i \geq 1}
\frac{1-u^i}{1- qu^i}= \frac{1}{1-uq} \prod_{i \geq 2}
\frac{1-u^i}{1-qu^i} >> \frac{1}{1-uq} \] where the last
step follows since $\frac{1- u^i}{1-qu^i} >>1$. In fact this argument shows
that the strict inequality $q^n < k(\AGL(n,q))$ holds for $n \geq 2$, since the coefficient of $u^i$ in $(1-u^i)/(1-qu^i)$ is positive.

For a second proof that $q^n \leq k(\AGL(n,q))$ with strict inequality if $n \geq 2$, note that $k(\GL(n,q))$ is at least $q^n - q^{n-1}$ and indeed is strictly greater for $n>1$, since there are $q^n-q^{n-1}$ semisimple classes (i.e. different characteristic polynomials) and for $n>1$, there are unipotent classes as well. Now use the fact (Lemma \ref{l:agl}) that \[ k(\AGL(n,q)) = 1 + \sum_{m=1}^n k(\GL(m,q)). \]

For the upper bound, we know from \cite{MR} that $k(\GL(m,q))<q^m$ for all $m$.
So again by Lemma \ref{l:agl},
 \[ k(\AGL(n,q)) \leq q^n + q^{n-1} + \cdots + 1 < 2 q^n.\]
\end{proof}

Finally, we give a result for $AH$ where $H$ is between $\GL$ and $\SL$.

\begin{theorem} Fix $q$ and let $\SL(n,q) \leq H = H(n,q) \leq \GL(n,q)$ with $e = [H:\SL(n,q)] < q-1$. Then $k(AH) < q^n$ except for $k(\ASL(1,q))=q$ and $k(\ASL(2,3))=10$.
\end{theorem}

\begin{proof} Suppose that $n=1$. Then as noted in Lemma \ref{bet}, $k(AH(1,q)) = e + (q-1)/e$. Now if $e + \frac{q-1}{e} \geq q$, then $e^2-1 \geq q(e-1)$. So either $e-1=0$ or $e + 1 \geq q$. But $e < q-1$ so the only remaining possibility is $n=1,e=1$, as claimed.

Now we suppose that $n \geq 2$. From \cite{FG1}, $k(H) \leq e \cdot k(SL(n,q))$. So from Lemma \ref{bet},
\[ k(AH) \leq \frac{q-1}{e} + e [k(\SL(1,q)) \cdots + k(\SL(n,q))] .\] From \cite{FG1}, $k(\SL(j,q)) \leq 2.5 q^{j-1}$. Thus
\[ k(AH) \leq \frac{q-1}{e} + 2.5 e \frac{q^n-1}{q-1}.\]

We claim that if $(q-1)/e \geq 3$, then
\[ \frac{q-1}{e} + 2.5 e \frac{q^n-1}{q-1} \leq q^n.\] Indeed, if $(q-1)/e \geq 3$,
then \[ \frac{q-1}{e} + 2.5 e \frac{q^n-1}{q-1} \leq \frac{q-1}{e} + (q^n-1) \frac{2.5}{3} .\]
Since $(q-1)/e \geq 3$, we have that $q \geq 4$, and it is easy to check that if $q \geq 4$, then \[ \frac{q-1}{e} + (q^n-1) \frac{2.5}{3} \leq q^n.\]

Since $e < q-1$, the remaining case is that $(q-1)/e = 2$. Since $(q-1)/e$ is even, we can assume that $q$ is odd. Then by Proposition 3.8 of \cite{FG1},
 \begin{equation} k(H) = \left\{ \begin{array}{ll}
\frac{1}{2} k(\GL(n,q)) & \mbox{if $n$ is odd} \\
\frac{1}{2} k(\GL(n,q)) + \frac{3}{2} k(\GL(n/2,q)) & \mbox{if $n$ is even}
\end{array} \right. \end{equation} Using the fact that $k(\GL(j,q))<q^j$ and Lemma \ref{bet}, one easily checks that if $q \geq 5$, then $k(AH) \leq q^n$. Similarly if $q=3$ (so $e=1$ and $H=\SL$), it is not hard to see that  $k(\ASL(2,3))=10$ and that $k(\ASL(n,3))<3^n$ otherwise. \end{proof}

\section{AGU and related groups} \label{unita}

Section \ref{Uorb} uses the orbit approach to calculate the generating function for $k(\AGU(n,q))$. Section \ref{Uchar} uses the character approach to calculate the generating function for $k(\AGU(n,q))$. Section \ref{GUbound} uses this generating function to obtain bounds on the number of conjugacy classes of $\AGU(n,q)$ and related groups.

\subsection{Orbit approach to $k(\AGU)$} \label{Uorb}

This section uses the orbit approach to calculate the generating function for $k(\AGU(n,q))$.

The following theorem calculates $o(C)$ for a conjugacy class $C$ of $\GU(n,q)$. This only involves $\lambda_{z-1}(C)$, the partition corresponding to the eigenvalue $1$ in the rational canonical form of the conjugacy class $C$. As in the $\GL$ case, let $d(\lambda)$ be the number of distinct parts of the partition $\lambda$. In what follows we also let $b(\lambda)$ denote the number of part sizes of $\lambda$ which have multiplicity exactly $1$.

\begin{theorem} \label{unitform} Let $C$ be a conjugacy class of $\GU(n,q)$. Then
\[ o(C) = 1 + q \cdot d(\lambda_{z-1}(C)) - b(\lambda_{z-1}(C)) .\]
\end{theorem}

\begin{proof} It suffices to assume that $C$ consists of unipotent elements and so corresponds to a partition $\lambda$.
The proof is similar to the case of $\GL$.

Now write $V = V_1 \oplus \ldots  \oplus V_m$ where $g|V_i$  has all Jordan blocks of size $i$.  We only consider the nonzero $V_i$.
So $ d_i = \dim V_i/[g,V_i]$ is the number of Jordan blocks of size $i$.  It is well known that the centralizer of $g$ induces the full
$\GU(d_i,q)$ and so there are $q$ orbits of the form $gv$ with $0 \ne v \in V_i$ for $d_i > 1$ and $q-1$ orbits if $d_i=1$ (there are no
nontrivial vectors of norm $0$ if $d_i=1$).

    Note that if $h \in C(g)$,  then $hV_i \subset  V_1 \oplus \ldots \oplus V_i + [g,V]$.
Thus, two elements in $V$ which are in the same $C(g)$-orbit module $[g,V]$ must have the same highest nonzero (modulo $[g,V]$) term.
Conversely, we need to show that any two such vectors are in the same orbit and indeed are in the orbit of $v_j$ with $v_j \in V_j \setminus [g,V_j]$.
By induction, we may assume that $j=m$.   Note that there exists $h \in C(g)$ so that $h$ is trivial on $V/\sum_{e < m} V_e$ and $hv_m - v_m$ is an arbitrary element in $\oplus_{e < m} V_e/[g,V_e]$.
Thus, we see that the  $v$ and $v_m$ are in the same orbit.  The number of orbits for the nontrivial $v_m$ is $q$ or $q-1$ as above.   The result follows.
\end{proof}

The following combinatorial lemma will also be helpful.

\begin{lemma} \label{genfunU}
\begin{enumerate}
\item The generating function for the number of unipotent classes of $\GU(n,q)$ is
defined as \[ \sum_{\lambda} u^{|\lambda|}.\] This is equal to
\[ \prod_i \frac{1}{1-u^i} .\]

\item The generating function \[ \sum_{\lambda} d(\lambda) u^{|\lambda|} \] is equal to \[ \frac{u}{1-u} \prod_i \frac{1}{1-u^i} .\]

\item The generating function \[ \sum_{\lambda} b(\lambda) u^{|\lambda|} \] is equal to \[ \frac{u}{1-u^2} \prod_i \frac{1}{1-u^i} .\]
\end{enumerate}
\end{lemma}

\begin{proof} The first part is just the well known generating function for the partition function. The second part is in the proof of Lemma \ref{distinct}.

For the third assertion, note that
\[ \sum_{\lambda} x^{b(\lambda)} u^{|\lambda|} \] is equal to
\[ \prod_i (1 + xu^i + u^{2i} + u^{3i} + \cdots).\] Differentiating with respect to $x$ and setting $x=1$ gives that
\[ \sum_{\lambda} b(\lambda) u^{|\lambda|} \] is equal to
\begin{eqnarray*}
\sum_i u^i \prod_{j \neq i} (1 + u^j + u^{2j} + u^{3j} + \cdots)
& = & \sum_i u^i \prod_{j \neq i} \frac{1}{1-u^j} \\
& = & \sum_i u^i (1-u^i) \prod_j \frac{1}{1-u^j} \\
& = & \frac{u}{1-u^2} \prod_j \frac{1}{1-u^j},
\end{eqnarray*} as claimed.
\end{proof}

Theorem \ref{exactU} gives an exact generating function for $k(\AGU(n,q))$.

\begin{theorem} \label{exactU} $k(\AGU(n,q))$ is equal to the coefficient of $u^n$ in
\[ \prod_i \frac{1+u^i}{1-qu^i} \cdot \left( 1 + \frac{qu^2 + (q-1)u}{1-u^2} \right).\]
\end{theorem}

\begin{proof} By Lemma \ref{exact} and Theorem \ref{unitform}, $k(\AGU(n,q))$ is equal to $T_1 + T_2 - T_3$ where $T_1$ is $k(\GU(n,q))$, and $T_2,T_3$ are the following sums over conjugacy classes of $\GU(n,q)$:
\[ T_2 = q \sum_C d(\lambda_{z-1}(C)) \]
\[ T_3 = \sum_C b(\lambda_{z-1}(C)) \]

From Wall \cite{W}, $T_1$ is the coefficient of $u^n$ in \[ \prod_i \frac{1+u^i}{1-qu^i}.\]

To compute the generating function of $T_2$, we take Wall's generating function for $T_1$, divide it by the generating function for unipotent conjugacy classes in part 1 of Lemma \ref{genfunU}, and multiply it by the weighted sum over unipotent classes in part 2 of Lemma \ref{genfunU}. We conclude that $T_2$ is the coefficient of $u^n$ in \[ \frac{qu}{1-u} \prod_i \frac{1+u^i}{1-qu^i}.\]

To compute the generating function of $T_3$, we take Wall's generating function for $T_1$, divide it by the generating function for unipotent conjugacy classes in part 1 of Lemma \ref{genfunU}, and multiply it by the weighted sum over unipotent classes in part 3 of Lemma \ref{genfunU}. We conclude that $T_3$ is the coefficient of $u^n$ in \[ \frac{u}{1-u^2} \prod_i \frac{1+u^i}{1-qu^i}.\]

Putting the pieces together, we conclude that $k(\AGU(n,q))$ is the coefficient of $u^n$ in
\[ \prod_i \frac{1+u^i}{1-qu^i} \cdot \left( 1 + \frac{qu}{1-u} - \frac{u}{1-u^2} \right),\] which simplifies to the desired result. \end{proof}

\subsection{Character approach to $k(\AGU)$} \label{Uchar}

We use Lemma \ref{l:semidirect} to find a recursion for $k(\AGU)$. Then we use this to compute the generating function for $k(\AGU)$, giving another proof of Theorem \ref{exactU}.

Recall that if $H$ is a finite group and $p$ is a prime, then $O_p(H)$ is the (unique) maximal normal $p$-subgroup of $H$.

\begin{lemma} \label{l:aun}
\begin{eqnarray*}
k(\AGU(n,q)) & = & k(\GU(n,q)) + (q-1)k(\GU(n-1,q)) \\
& & + k(\AGU(n-2,q)) + (q-1)k(\GU(n-2,q)).
\end{eqnarray*}
\end{lemma}

\begin{proof}  We use the convention that $\GU(0,q)$ and $\AGU(0,q)$ are trivial groups and that $\GU(-1,q)$ and $\AGU(-1,q)$ are the empty set.
We can identify the natural module and the character group of the module because the module is self dual viewed over the field of $q$-elements.

Note that $\AGU(1,q)$ is a semidirect product of an elementary abelian group of order $q^2$ and $\GU(1,q)$ which is cyclic of order $q+1$.
Thus, it follows that $k(\AGU(1,q)) = k(\GU(1,q)) + (q-1)$ as claimed.    If $n=2$, we note that $\GU(2,q)$ has precisely $q$ nontrivial orbits on the natural module.
The stabilizer of a nondegenerate vector is $\GU(1,q)$ and the stabilizer of a totally singular vector is elementary abelian of order $q$ and again we
see the result holds.

Now suppose that $n \geq 3$.  Thus, we see that there are $q-1$ orbits with
stabilizer isomorphic to $\GU(n-1,q)$ (corresponding to vectors with a given nonzero norm) and the stabilizer
$H$ of a singular vector.   Note that $H$ has a center $Z$ of order $q$ and $H/Z \cong \AGU(n-2,q)$.   Also note
that any irreducible character of $U = O_p(H)$  that is nontrivial on $Z$  has dimension $q^{n-2}$ and corresponds to one of the $q-1$
nontrivial $1$-dimensional characters on $Z$.    Moreover each of these representations extends to a representation  of $H$
(this can be seen by considering the normalizer of $U$ in the full linear group).  Fix a nontrivial linear character of $Z$
an irreducible module $W$ of $H$ that affords this linear representation.  It follows by Clifford theory \cite[51.7]{CR} that
any irreducible representation of $H$ nontrivial on $Z$ is of the form
$W \otimes W'$ where  $W'$ is an irreducible
$H/U$-module.  Since there are $q-1$  nontrivial central characters of $U$ and there are $k(\GU(n-2,q))$
choices for $W'$, the result follows.
\end{proof}

We now give a second proof of Theorem \ref{exactU}.

\begin{proof} Let $k_n=k(\GU(n,q))$ and let $a_n=k(\AGU(n,q))$. Then Lemma \ref{l:aun} gives \begin{equation} \label{recu}
a_n = k_n + (q-1) k_{n-1} + (q-1)k_{n-2} + a_{n-2}.
\end{equation}

Let \[ K(u) = 1 + \sum_{n \geq 1} k_n u^n \ \ , \ \ A(u) = 1 + \sum_{n \geq 1} a_n u^n .\]

Multiplying \eqref{recu} by $u^n$ and summing over $n \geq 1$ gives that

\begin{eqnarray*}
A(u) - 1 & = & K(u) - 1 + (q-1) u K(u) \\
& & + (q-1) u^2 K(u) + u^2 A(u).
\end{eqnarray*}

Solving for $A(u)$, one obtains that
\begin{eqnarray*}
A(u) & = & K(u) \left( \frac{1+u(q-1)+u^2(q-1)}{1-u^2} \right) \\
& = & K(u) \left( 1 + \frac{qu^2+(q-1)u}{1-u^2} \right).
\end{eqnarray*}

From Wall \cite{W}, \[ K(u) = \prod_i \frac{1+u^i}{1-qu^i},\] and the theorem follows.
\end{proof}

\subsection{Bounds for $AGU$ and related groups} \label{GUbound}

As a corollary, we obtain the following result.

\begin{cor} \[ k(\AGU(n,q)) \leq 20 q^n.\]
\end{cor}

\begin{proof} From Theorem \ref{exactU} , $k(\AGU(n,q))$ is equal to the coefficient of $u^n$ in \[\prod_i \frac{1-u^i}{1-qu^i} \prod_i \frac{1+u^i}{1-u^i} \left( 1 + \frac{qu^2 + (q-1)u}{1-u^2} \right).\] Now all coefficients of powers of $u$ in
\[ \prod_i \frac{1+u^i}{1-u^i} \left( 1 + \frac{qu^2 + (q-1)u}{1-u^2} \right) \] are non-negative. It follows that $k(\AGU(n,q))$ is at most
\begin{eqnarray*}
& & \sum_{m=0}^n \rm{Coef.} \ u^{n-m} \ \rm{in} \ \prod_i \frac{1-u^i}{1-qu^i} \\
& & \cdot \rm{Coef.} \ u^m \ \rm{in} \ \prod_i \frac{1+u^i}{1-u^i} \left( 1 + \frac{qu^2 + (q-1)u}{1-u^2} \right).
\end{eqnarray*}

Now $\prod_i \frac{1-u^i}{1-qu^i}$ is the generating function for the number of conjugacy classes of $\GL(n,q)$. By \cite{MR}, $k(\GL(n,q))$ is at most $q^n$. Hence the coefficient of $u^{n-m}$ in it is at most $q^{n-m}$. It follows that $k(\AGU(n,q))$ is at most \[ q^n \sum_{m=0}^n \frac{1}{q^m} (\rm{Coef.} \ u^m \ \rm{in} \ \prod_i \frac{1+u^i}{1-u^i} \left( 1 + \frac{qu^2 + (q-1)u}{1-u^2} \right)).\] Since the coefficients of $u^m$ in \[ \prod_i \frac{1+u^i}{1-u^i} \left( 1 + \frac{qu^2 + (q-1)u}{1-u^2} \right), \] are non-negative, it follows that $k(\AGU(n,q))$ is at most \[ q^n \sum_{m=0}^{\infty} \frac{1}{q^m} (\rm{Coef.} \ u^m \  \rm{in} \ \prod_i \frac{1+u^i}{1-u^i} \left( 1 + \frac{qu^2 + (q-1)u}{1-u^2} \right)), \] which (set $u=1/q$) is equal to \[ q^n \prod_i \frac{(1+1/q^i)}{(1-1/q^i)} \cdot \left( 1 + \frac{1}{1-1/q^2} \right).\] The term \[ \prod_i \frac{(1+1/q^i)}{(1-1/q^i)} \cdot \left( 1 + \frac{1}{1-1/q^2} \right) \] is visibly maximized among prime powers $q$ when $q=2$, when it is at most $20$ (we used the remark after Lemma \ref{pentag} to bound the infinite product).
\end{proof}

\begin{cor} $k(\AGU(n,q)) \leq q^{2n}$.
\end{cor}

\begin{proof} By the preceding result, this holds if $20 \leq q^n$. So we only need to check the cases $n=1$, or $n=2,q=2,3,4$ or $n=3,q=2$ or $n=4,q=2$. From the generating function (Theorem \ref{exactU}), $k(\AGU(1,q))=2q$, and the other finite number of cases are computed easily from the generating function and seen to be at most $q^{2n}$.
\end{proof}

We can also use the previous results to get bounds for the groups between $\ASU(n,q)$ and $\AGU(n,q)$.  Since $\SL(2,q) \cong \SU(2,q)$, we assume that
$n \ge 3$.   With more effort one can get much better bounds as we did in the case of $\SL(n,q)$.   We just obtain the bound required for the $k(GV)$ problem.

\begin{cor}  Let $n \ge 3$.   Let $\ASU(n,q) \le H \le \AGU(n,q)$.
Then $k(H) \le q^{2n}$.
\end{cor}

\begin{proof}   Let $G=\AGU(n,q)$.   Since $[G:H]\ \le q+1$,  $k(H) \le  k(G)(q+1) \le 20q^n(q+1)$.   This is at most $q^{2n}$ unless $q=2$ with $n \le 5$
or $q=3$ or $4$ and $n=3$. These cases all follow using the exact values of $k(G)$ (obtained from our generating function) in the bound $k(H) \le  k(G)(q+1)$, except for the cases $q=2, n=3,4$. One computes (either using a recursion similar to Lemma \ref{l:aun} and exact values of $k(\SU)$ in \cite{M}, or by Magma) that $k(\ASU(3,2))=24$ and $k(\ASU(4,2))=49$, completing the proof.
\end{proof}

\section{ASp} \label{Symplec}

Section \ref{Sporb} uses the orbit approach to calculate the generating function for $k(\ASp(2n,q))$, assuming that the characteristic is odd. Section \ref{Spchar} uses the character approach to calculate the generating function for $k(\ASp(2n,q))$ in both odd and even characteristic. Section \ref{Spbound} uses these generating functions to obtain bounds on $k(\ASp(2n,q))$.

\subsection{Orbit approach to $k(\ASp)$, odd characteristic } \label{Sporb}

This section treats the affine symplectic groups. We only work in odd characteristic.  In this case the conjugacy class of a unipotent element is
determined by its Jordan form (over the algebraic closure) and it is much more complicated to deal with the characteristic $2$ case.  Since our
character approach works in characteristic $2$, we will not pursue the direct approach in that case.   So for this section, let $q$ be odd.

The following theorem calculates $o(C)$ for a conjugacy class $C$ of $\Sp(2n,q)$.
This only involves the unipotent part of the class $C$. Recall that the conjugacy class of a unipotent element is determined (over the algebraic
closure) by a partition of $2n$ with $a_i$ parts of $i$.  Moreover,  $a_i$ is even if $i$ is odd.  Over a finite field,
we attach a sign $\epsilon_i$ for each even $i$ with $a_i \ne 0$ and this gives a description of all the unipotent conjugacy
classes (see \cite{LS} for details). We let $\lambda^{\pm}_{z-1}(C)$ denote this signed partition for the unipotent part of the class $C$.

\begin{theorem} \label{sympformodd} Suppose that the characteristic is odd. Let $C$ be a conjugacy class of $\Sp(2n,q)$.
Let $a_i$ be the number of parts of $\lambda^{\pm}_{z-1}(C)$ of size $i$. Then $o(C)$ is equal to
\[ 1 +  \sum_{i \ odd \atop a_i \neq 0} 1 + \sum_{i \ even \atop a_i \neq 0} f_i \]
where \begin{equation} \label{fort} f_i = \left\{ \begin{array}{ll}
q & \mbox{if $a_i > 2$ (independently of the sign)} \\
q & \mbox{if $a_i = 2$ and the sign is $+$} \\
(q-1) & \mbox{if $a_i = 2$ and the sign is $-$} \\
(q-1)/2 & \mbox{if $a_i = 1$ (independently of the sign)}
\end{array} \right. \end{equation}
\end{theorem}

\begin{proof}  The proof is similar to the case of $\GL$ and $\GU$ and reduces to the case of unipotent elements.  So assume
that $C$ is a unipotent class.  Let $g \in C$.  Write $V$ as an orthogonal direct sum of spaces $V_i$ where $g$ has $a_i$ Jordan
blocks  of size $i$ on $V_i$.   As in the previous cases, one can show that $gv$ is either conjugate to $g$ or for some $i$, $g$
is conjugate to $gv_i$ where $v_i \in V_i \setminus{[g, V_i]}$.

By \cite{LS},  we see that the there is a subgroup of $C(g)$ acting as $\Sp(a_i,q)$ for $i$ odd or
$\OO^{\epsilon_i}(a_i,q)$ if $i$ is even acting naturally on   $V_i/[g,V_i]$.   Thus, the number of classes
of the form $gv_i$ with $vI \in V_i \setminus{[g, V_i]}$ is $1$ is if $i$ is odd and $f_i$ as given above if $i$ is even.
\end{proof}

The following combinatorial lemma will also be helpful.

\begin{lemma} \label{genfun} Suppose that the characteristic is odd.

\begin{enumerate}
\item The generating function for the number of unipotent classes of the groups $\Sp(2n,q)$ is defined as \[ \sum_{\lambda^{\pm}} u^{|\lambda^{\pm}|/2}.\] This is equal to
     \[ \prod_{i \ odd} \frac{1}{1-u^i} \prod_i \left( \frac{1+u^i}{1-u^i} \right). \]

\item The generating function \[ \sum_{\lambda^{\pm}} u^{|\lambda^{\pm}|/2}
 \sum_{j \ odd \atop a_j \neq 0} 1
 \] is equal to \[ \frac{u}{1-u^2} \prod_{i \ odd} \frac{1}{1-u^i} \prod_i \left( \frac{1+u^i}{1-u^i} \right).\]

\item Let $f_j$ be as in Theorem \ref{sympformodd}. The generating function \[ \sum_{\lambda^{\pm}} u^{|\lambda^{\pm}|/2} \sum_{j \ even \atop a_j \neq 0} f_j \] is equal to \[ \left( \frac{(q-1)u}{1-u} + \frac{u^2}{1-u^2} \right) \prod_{i \ odd} \frac{1}{1-u^i} \prod_i \left( \frac{1+u^i}{1-u^i} \right) \]

\end{enumerate}
\end{lemma}

\begin{proof} For the first part, the unipotent conjugacy classes of $\Sp(2n,q)$ correspond to singed partitions $\lambda^{\pm}$ of size $2n$. Clearly the generating function for such partitions
is equal to
\[ \prod_{i \ odd} (1+u^i+u^{2i}+\cdots) \prod_{i \ even} (1+2u^{i/2} + 2u^{2i/2}+\cdots) \]
which is equal to \[ \prod_{i \ odd} \frac{1}{1-u^i} \prod_i \left( 1 + \frac{2u^i}{1-u^i} \right) = \prod_{i \ odd} \frac{1}{1-u^i} \prod_i \left( \frac{1+u^i}{1-u^i} \right).\]

For the second part, first note that arguing as in the first part, one has that
\[ \sum_{\lambda^{\pm}} u^{|\lambda^{\pm}|/2} \sum_{j \ odd \atop a_j \neq 0} 1 \]
is equal to
\begin{eqnarray*}
\sum_{j \ odd} \sum_{\lambda^{\pm} \atop a_j \neq 0} u^{|\lambda^{\pm}|/2} & = & \sum_{j \ odd} (u^j+u^{2j}+\cdots) \prod_{i \ odd \atop i \neq j} (1+u^i+u^{2i}+\cdots) \\
& & \cdot \prod_{i \ even} (1+2u^{i/2}+2u^{2i/2}+\cdots) \\
& = & \sum_{j \ odd} u^j \prod_{i \ odd} (1+u^i+u^{2i}+\cdots) \\
& & \cdot \prod_{i \ even} (1+2u^{i/2}+2u^{2i/2}+\cdots) \\
& = & \frac{u}{1-u^2} \prod_{i \ odd} \frac{1}{1-u^i} \prod_i \left( \frac{1+u^i}{1-u^i} \right).
\end{eqnarray*}

For the third part, \[ \sum_{\lambda^{\pm}} u^{|\lambda^{\pm}|/2} \sum_{j \ even \atop a_j \neq 0} f_j \] is equal to
\begin{eqnarray*}
& & \sum_{j \ even} \left( 2 u^{j/2} \frac{(q-1)}{2} + u^{2j/2} (q+q-1) + 2q(u^{3j/2}+u^{4j/2}+\cdots) \right) \\
& & \cdot \prod_{i \ odd} (1+u^i+u^{2i}+\dots) \prod_{i \ even \atop i \neq j} \frac{1+u^{i/2}}{1-u^{i/2}}.
\end{eqnarray*}

This is equal to
\begin{eqnarray*}
& & \sum_{j \ even} \frac{1-u^{j/2}}{1+u^{j/2}} \left( u^{j/2}(q-1) + u^{2j/2} (2q-1) + 2q (u^{3j/2}+u^{4j/2}+ \cdots) \right) \\
& & \cdot \prod_{i \ odd} \frac{1}{1-u^i} \prod_i \frac{1+u^i}{1-u^i}
\end{eqnarray*}

Now clearly
\[ \sum_{j \ even} \frac{1-u^{j/2}}{1+u^{j/2}} \left( u^{j/2}(q-1) + u^{2j/2} (2q-1) + 2q (u^{3j/2}+u^{4j/2}+ \cdots) \right) \] is equal to
\begin{eqnarray*}
& & \sum_j \frac{1-u^j}{1+u^j} \left( u^j(q-1) + u^{2j} (2q-1) + 2q (u^{3j}+u^{4j}+ \cdots) \right) \\
& = & \sum_j \frac{1}{1+u^j} \left( qu^j - u^j + qu^{2j} + u^{3j} \right) \\
& = & \sum_j (qu^j + u^{2j} - u^j) \\
& = & \frac{(q-1)u}{1-u} + \frac{u^2}{1-u^2},
\end{eqnarray*} and the third part of the lemma follows.
\end{proof}

\begin{theorem} \label{SpGenOdd} In odd characteristic, $k(\ASp(2n,q))$ is equal to the coefficient of $u^n$ in \[ \prod_i \frac{(1+u^i)^4}{1-qu^i} \cdot \left( 1 + \frac{qu}{1-u} \right).\]
\end{theorem}

\begin{proof} By Lemma \ref{exact} and Theorem \ref{sympformodd}, $k(\ASp(2n,q))$ is equal to $T_1 + T_2 + T_3$, where $T_1$ is $k(\Sp(2n,q))$, and $T_2,T_3$ are the following sums over conjugacy classes of $\Sp(2n,q)$:

\[ T_2 =  \sum_C \sum_{i \ odd \atop a_i \neq 0} 1 \]
\[ T_3 = \sum_C \sum_{i \ even \atop a_i \neq 0} f_i\]

From Wall \cite{W}, $T_1$ is the coefficient of $u^n$ in \[ \prod_i \frac{(1+u^i)^4}{1-qu^i}.\]

To compute the generating function of $T_2$, we take Wall's generating function for $T_1$, divide it by the generating function for unipotent conjugacy classes in part 1 of Lemma \ref{genfun}, and multiply it by the generating function for the weighted sum over unipotent classes in part 2 of Lemma \ref{genfun}.
We conclude that $T_2$ is the coefficient of $u^n$ in
\[ \frac{u}{1-u^2} \prod_i \frac{(1+u^i)^4}{1-qu^i}.\]

To compute the generating function of $T_3$, we take Wall's generating function for $T_1$, divide it by the generating function for unipotent conjugacy classes in part 1 of Lemma \ref{genfun}, and multiply it by the generating function for the weighted sum over unipotent classes in part 3 of Lemma \ref{genfun}.
We conclude that $T_3$ is the coefficient of $u^n$ in
\[ \left( \frac{(q-1)u}{1-u} + \frac{u^2}{1-u^2} \right) \prod_i \frac{(1+u^i)^4}{1-qu^i}.\]

Since
\[ 1 + \frac{u}{1-u^2}+\frac{(q-1)u}{1-u} + \frac{u^2}{1-u^2} = 1 + \frac{qu}{1-u},\]
the proof of the theorem is complete.
\end{proof}

\subsection{Character approach to $k(\ASp(2n,q))$, any characteristic} \label{Spchar}

As in the other cases, we apply Lemma \ref{l:semidirect}.

To begin we treat the case of odd characteristic.

\begin{lemma} \label{l:aspodd}  Let $q$ be odd and $G=\Sp(2n,q)$.  Then $k(AG)  = k(\Sp(2n,q)) + k(\ASp(2n-2,q)) + (q-1) k(\Sp(2n-2,q))$.
\end{lemma}

\begin{proof}  We take    $\ASp(0,q)$ and $\Sp(0,q)$ to be the trivial group.  If $n=1$, then $G=\SL(2,q)$.
It is straightforward to see that $k(\SL(2,q)) = q + 4$ and that $k(\ASL(2,q)) =  2q+4$ and so the formula holds.

So suppose that $n \geq 2$.  Let $V$ be the natural module for $G$.  Note that in this case $G$ acts transitively on the nontrivial characters of $V$
and the stabilizer of such a character is the stabilizer  $H$ of a vector in $\Sp(2n,q)$.
Let $U = O_p(H)$ and let $Z=Z(H)$.   Then  $H/Z \cong \ASp(2n-2,q)$.  If an irreducible character
of $H$ does not vanish on $Z$, then there are $q-1$ possibilities (depending on the restriction to $Z$)
and arguing as in the unitary case, we see that the number of such characters of $H$ is $(q-1)k(\Sp(2n-2,q))$.
This gives $k(\ASp(2n,q)) =k(\Sp(2n,q)) +  k(\ASp(2n-2,q)) + (q-1) k(\Sp(2n-2,q))$ as desired.
\end{proof}

We use this recursion to give another proof of the generating function for $k(\Sp(2n,q))$ in odd characteristic.

\begin{proof} (Second proof of Theorem \ref{SpGenOdd})
Let $k_n = k(\Sp(2n,q))$ and let $a_n = k(\ASp(2n,q))$. Lemma \ref{l:aspodd} gives that
\begin{equation} \label{symprecur}
a_n = k_n + (q-1) k_{n-1} + a_{n-1}
\end{equation}

Let \[ K(u) = 1 + \sum_{n \geq 1} k_n u^n \ \ , \ \ A(u) = 1 + \sum_{n \geq 1} a_n u^n.\]

Multiplying \eqref{symprecur} by $u^n$ and summing over $n \geq 1$ gives that
\[ A(u) - 1 = K(u) - 1 + (q-1) u K(u) + u A(u).\]

Solving for $A(u)$ gives

\begin{eqnarray*}
A(u) & = & K(u) \left( \frac{1+u(q-1)}{1-u} \right) \\
& = & K(u) \left( 1 + \frac{qu}{1-u} \right).
\end{eqnarray*}

From Wall \cite{W}, \[ K(u) = \prod_i \frac{(1+u^i)^4}{1-qu^i}, \] and the result
follows.
\end{proof}

In even characteristic, the unipotent radical is abelian but not irreducible.   So let $G = \Sp(2n,q)$ with $q$ even.
Let $BG$ denote the semidirect product $WG$ where $W$ is the $2n+1$ dimensional indecomposable module with
$G$ having a one dimensional fixed space $W_0$ and $W/W_0 \cong V$.

Note that the $G$-orbits of characters of $B$ consist of the trivial character,  one orbit of nontrivial characters with
$W_0$ contained in the kernel and $2(q-1)$ orbits of characters with are nontrivial on $W_0$.  The stabilizer of
a character in the second orbit is isomorphic to $B\Sp(2n-2,q)$ while in the final case the stabilizers are
$\OO^{\pm}(2n,q)$ (with $q-1$ of each type). This gives the following:

\begin{lemma} \label{l:aspeven}   Let $q$ be even.
\begin{enumerate}
\item  $k(B\Sp(2n,q)) = k(\ASp(2n,q))  + (q-1)(k(\OO^+(2n,q)) + k(\OO^-(2n,q)))$
\item   $k(\ASp(2n,q)) = k(\Sp(2n,q)) + k(B\Sp(2n-2,q))$.
\end{enumerate}
\end{lemma}

The next lemma follows immediately from the previous lemma. We use the convention that $\ASp(0,q)$ and $\OO^+(0,q)$ are the trivial groups and that $\OO^-(0,q)$ is the empty set. So $k(\ASp(0,q))=1$, $k(\OO^+(0,q))=1$, and $k(O^-(0,q))=0$.

\begin{lemma} \label{recur} For all $n \geq 1$,
\begin{eqnarray*}
k(\ASp(2n,q)) & = & k(\Sp(2n,q)) + k(\ASp(2n-2,q)) \\
& & + (q-1) [k(\OO^+(2n-2,q)) + k(\OO^-(2n-2,q))].
\end{eqnarray*}
\end{lemma}

Now we obtain the generating function for $k(\ASp(2n,q))$ in even characteristic.

\begin{theorem} \label{SpGenEven} In even characteristic, $k(\ASp(2n,q))$ is equal to the coefficient of $u^n$ in
\begin{eqnarray*}
& & \frac{1}{1-u} \prod_i \frac{1+u^i}{1-qu^i} \\
& & \cdot \left[ \prod_i \frac{1}{(1-u^{4i-2})^2} + (q-1)u \prod_i (1+u^{2i-1})^2 \right]
\end{eqnarray*}
\end{theorem}

\begin{proof} We define three generating functions:

\[ K_{Sp}(u) = 1 + \sum_{n \geq 1} k(\Sp(2n,q)) u^n \]
\[ K_O(u) = 1 + \sum_{n \geq 1} [k(\OO^+(2n,q)) + k(\OO^-(2n,q))] u^n \]
\[ A(u) = 1 + \sum_{n \geq 1} k(\ASp(2n,q)) u^n. \]

Multiplying the recursion from Lemma \ref{recur} by $u^n$ and summing over $n \geq 1$ gives that
\[ A(u)-1 = K_{Sp}(u) - 1 + u A(u) + (q-1) u K_O(u).\] Thus
\[ A(u) = \frac{K_{Sp}(u) + (q-1) u K_O(u)}{1-u} .\]

From Theorems 3.13 and Theorem 3.21 of \cite{FG1}, elementary manipulations, give that
\[ K_{Sp}(u) = \prod_i \frac{1+u^i}{1-qu^i} \prod_i \frac{1}{(1-u^{4i-2})^2} \]
\[ K_O(u) = \prod_i \frac{1+u^i}{1-qu^i} \prod_i (1+u^{2i-1})^2, \] and the result follows.
\end{proof}

\subsection{Bounds on $k(\ASp(2n,q))$} \label{Spbound}

As a corollary, we obtain the following results.

\begin{cor} \label{oddbounder} In odd characteristic, $k(\ASp(2n,q)) \leq 27 q^n$.
\end{cor}

\begin{proof} From Theorem \ref{SpGenOdd}, $k(\ASp(2n,q))$ is the coefficient of $u^n$ in
\[ \prod_i \frac{1-u^i}{1-qu^i} \prod_i \frac{(1+u^i)^4}{1-u^i}
\left( 1 + \frac{qu}{1-u} \right).\] Now all coefficients of powers of $u$ in \[ \prod_i \frac{(1+u^i)^4}{1-u^i}
\left( 1 + \frac{qu}{1-u} \right) \] are non-negative. It follows that $k(\ASp(2n,q))$ is at most
\begin{eqnarray*}
& & \sum_{m=0}^n (\rm{Coef.} \ u^{n-m} \ \rm{in} \ \prod_i \frac{1-u^i}{1-qu^i}) \\
& & \cdot (\rm{Coef.} \ u^m \ \rm{in} \ \prod_i \frac{(1+u^i)^4}{1-u^i} \cdot
\left( 1 + \frac{qu}{1-u} \right) ).
\end{eqnarray*}

Now $\prod_i \frac{1-u^i}{1-qu^i}$ is the generating function for the number of conjugacy classes in $\GL(n,q)$. By \cite{MR}, $k(\GL(n,q))$ is at most $q^n$. Hence the coefficient of $u^{n-m}$ in it is at most $q^{n-m}$. It follows that $k(\ASp(2n,q))$ is at most \[ q^n \sum_{m=0}^n \frac{1}{q^m} (\rm{Coef.} \ u^m \ \rm{in} \ \prod_i \frac{(1+u^i)^4}{1-u^i}
\left( 1 + \frac{qu}{1-u} \right) ).\] Since the coefficients of $u^m$ in \[ \prod_i \frac{(1+u^i)^4}{1-u^i}
\left( 1 + \frac{qu}{1-u} \right) \] are non-negative, it follows that $k(\ASp(2n,q))$ is at most  \[ q^n \sum_{m=0}^{\infty} \frac{1}{q^m} (\rm{Coef.} \ u^m \ \rm{in} \ \prod_i \frac{(1+u^i)^4}{1-u^i}
\left( 1 + \frac{qu}{1-u} \right) ), \] which is equal to
\[ q^n \prod_i \frac{(1+1/q^i)^4}{1-1/q^i}
\left( 1 + \frac{1}{1-1/q} \right).\] The term \[ \prod_i \frac{(1+1/q^i)^4}{1-1/q^i} \left( 1 + \frac{1}{1-1/q} \right) \] is visibly maximized among odd prime powers $q$ when $q=3$, when it is at most $27$ (we bounded the infinite product $\prod_i \frac{(1+1/q^i)^4}{1-1/q^i}$ using the remark after Lemma \ref{pentag}).
\end{proof}

\begin{cor} In odd characteristic,
\[ k(\ASp(2n,q)) \leq q^{2n},\] except for $k(\ASp(2,3))=10$.
\end{cor}

\begin{proof} From the previous result, $k(\ASp(2n,q)) \leq 27 q^n$. This immediately implies that $k(\ASp(2n,q)) \leq q^{2n}$ except possibly for $\ASp(2,q)$, $\ASp(4,3)$, or $\ASp(4,5)$.

From our generating function for $k(\ASp(2n,q))$ (Theorem \ref{SpGenOdd}), we see that $k(\ASp(4,3))=58$, $k(\ASp(4,5))=110$, and $k(\ASp(2,q)) = 2q+4$, and the result follows.
\end{proof}

Next we move to even characteristic.

\begin{cor} \label{Spevbound} In even characteristic, $k(\ASp(2n,q)) \leq 56 q^n$.
\end{cor}

\begin{proof} We rewrite the generating function for $k(\ASp(2n,q))$ in Theorem \ref{SpGenEven} as \begin{eqnarray*}
& & \prod_i \frac{1-u^i}{1-qu^i} \frac{1}{1-u} \prod_i \frac{1+u^i}{1-u^i} \\
& & \cdot \left[ \prod_i \frac{1}{(1-u^{4i-2})^2} + (q-1)u \prod_i (1+u^{2i-1})^2 \right]
\end{eqnarray*}

Now arguing exactly as in the odd characteristic case (Corollary \ref{oddbounder}), one sees that $k(\ASp(2n,q))$ is at most
\[ q^n \cdot \frac{1}{1-1/q} \prod_i \frac{1+1/q^i}{1-1/q^i} \left[ \prod_i \frac{1}{(1-1/q^{4i-2})^2} + (1-1/q) \prod_i (1+1/q^{2i-1})^2 \right], \] and the result follows.
\end{proof}

Next we classify when $k(\ASp(2n,q)) \leq q^{2n}$.

\begin{cor} In even characteristic, \[ k(\ASp(2n,q)) \leq q^{2n}, \] except for $k(\ASp(2,2))=5, k(\ASp(4,2))=21, k(\ASp(6,2))=67$.
\end{cor}

\begin{proof} From the previous result, $k(\ASp(2n,q)) \leq 56 q^n$. This immediately implies that $k(\ASp(2n,q)) \leq q^{2n}$ except possibly for $q=2, 1 \leq n \leq 5$, or $q=4,n=1,2$ or $q=8,n=1$. For these $q,n$ values one calculates $k(\ASp(2n,q))$ from the generating function in Theorem \ref{SpGenEven}, and the result follows.
\end{proof}

\section{Orthogonal Groups} \label{Orth}

Subsection \ref{orbitO} uses the orbit approach to calculate the generating function for $k(\AO)$ when the characteristic is odd. Subsection \ref{repO} uses the character approach to calculate the generating function of $k(\AO)$ in any characteristic.
To be more precise, we actually derive two generating functions, one for $k(\AO^+) + k(\AO^-)$ and one for $k(\AO^+) - k(\AO^-)$. Clearly this is equivalent to deriving generating functions for $k(\AO^+)$ and $k(\AO^-)$.

Section \ref{Obounds} derives some bounds on $k(\AO)$.

\subsection{Orbit approach for $k(\AO)$, odd characteristic} \label{orbitO}

For the orbit approach we assume the characteristic is odd.  It is somewhat more convenient to work in orthogonal groups than the special orthogonal group
(there is essentially no difference in the result below for $\SO$).
The conjugacy class of a unipotent element  $\OO^{\epsilon}(m,q)$ gives rise to a partition of $m$ with $a_i$ pieces of size $i$.   Moreover,
$a_i$ is even for $i$ even.   This determines the conjugacy class over the algebraic closure.  Over the finite field, we attach a sign $\epsilon_i$
for each odd $i$ with $a_i$ nonzero and this determines the class (see \cite{LS}). We let $\lambda^{\pm}_{z-1}(C)$ denote this signed partition corresponding to the unipotent part of a conjugacy class $C$.

The proof of the next result is essentially identical to the case of symplectic groups and so we omit the details (and we can also
use the character theory approach below).

\begin{theorem} \label{orthoformodd} Suppose that the characteristic is odd. Let $C$ be a conjugacy class of $\OO^{\epsilon}(n,q)$.
Let $a_i$ be the number of parts of $\lambda^{\pm}_{z-1}(C)$ of size $i$. Then $o(C)$ is equal to
\[ 1 +  \sum_{i \ even \atop a_i \neq 0} 1 + \sum_{i \ odd \atop a_i \neq 0} f_i \]
where \begin{equation} f_i = \left\{ \begin{array}{ll}
q & \mbox{if $a_i > 2$ (independently of the sign)} \\
q & \mbox{if $a_i = 2$ and the sign is $+$} \\
(q-1) & \mbox{if $a_i = 2$ and the sign is $-$} \\
(q-1)/2 & \mbox{if $a_i = 1$ (independently of the sign)}
\end{array} \right. \end{equation}
\end{theorem}

The following combinatorial lemma will also be helpful.

\begin{lemma} \label{genfunO} Suppose that the characteristic is odd.
\begin{enumerate}
\item The generating function for the number of unipotent classes of the groups
$\OO(n,q)$ is defined as \[ \sum_{\lambda^{\pm}} u^{|\lambda^{\pm}|}. \] This is equal to \[  \prod_i \frac{1}{1-u^{4i}} \prod_{i \ odd} \left( \frac{1+u^i}{1-u^i} \right).\]

\item The generating function \[ \sum_{\lambda^{\pm}} u^{|\lambda^{\pm}|} \sum_{j \ even \atop a_j \neq 0} 1 \] is equal to  \[ \frac{u^4}{1-u^4} \prod_i \frac{1}{1-u^{4i}} \prod_{i \ odd} \left( \frac{1+u^i}{1-u^i} \right).\]

\item Let $f_i$ be as in Theorem \ref{orthoformodd}. Then \[ \sum_{\lambda^{\pm}} u^{|\lambda^{\pm}|} \sum_{j \ odd \atop a_j \neq 0} f_j \] is equal to \[ \left( \frac{(q-1)u}{1-u^2} + \frac{u^2}{1-u^4} \right) \prod_i \frac{1}{1-u^{4i}} \prod_{i \ odd} \left( \frac{1+u^i}{1-u^i} \right).\]
\end{enumerate}
\end{lemma}

\begin{proof} For the first part, the unipotent conjugacy classes of the groups $\OO(n,q)$ correspond to signed partitions $\lambda^{\pm}$ of size $n$. The generating function for such partitions is clearly equal to
\[ \prod_{i \ odd} (1 + 2u^i + 2u^{2i} + \cdots) \prod_{i \ even} (1+u^{2i}+u^{4i}+ \cdots), \] which is equal to
\[ \prod_i \frac{1}{1-u^{4i}} \prod_{i \ odd} \frac{1+u^i}{1-u^i}.\]

For the second part, first note that arguing as in the first part, one has that \[ \sum_{\lambda^{\pm}} u^{|\lambda^{\pm}|} \sum_{j \ even \atop a_j \neq 0} 1 \] is equal to
\begin{eqnarray*}
\sum_{j \ even} \sum_{\lambda^{\pm} \atop a_j \neq 0} u^{|\lambda^{\pm}|}
& = & \sum_{j \ even} (u^{2j}+u^{4j}+\cdots) \prod_{i \ even \atop i \neq j}
(1+u^{2i}+u^{4i}+\cdots) \\
& & \cdot \prod_{i \ odd} (1+2u^i+2u^{2i}+\cdots) \\
& = & \sum_{j \ even} u^{2j} \prod_{i \ even}
(1+u^{2i}+u^{4i}+\cdots) \\
& & \cdot \prod_{i \ odd} (1+2u^i+2u^{2i}+\cdots) \\
& = & \frac{u^4}{1-u^4} \prod_i \frac{1}{1-u^{4i}} \prod_{i \ odd} \left( \frac{1+u^i}{1-u^i} \right).
\end{eqnarray*}

For the third part, \[ \sum_{\lambda^{\pm}} u^{|\lambda^{\pm}|} \sum_{j \ odd
\atop a_j \neq 0} f_j \] is equal to
\begin{eqnarray*}
& & \sum_{j \ odd} \left( 2u^j \frac{(q-1)}{2} + u^{2j}(q+q-1) + 2q(u^{3j}+u^{4j} + \cdots)   \right) \\
& & \cdot \prod_{i \neq j \atop i \ odd} \left( \frac{1+u^i}{1-u^i} \right)
\prod_{i \ even} (1+u^{2i}+u^{4i}+\cdots). \end{eqnarray*}

This is equal to
\begin{eqnarray*}
& & \sum_{j \ odd} \frac{1-u^j}{1+u^j} \left( u^j (q-1)  + u^{2j}(2q-1) + 2q(u^{3j}+u^{4j} + \cdots) \right) \\
& & \cdot \prod_{i \ odd} \left( \frac{1+u^i}{1-u^i} \right)
\prod_i \frac{1}{1-u^{4i}}. \end{eqnarray*}

Now, as in the proof of part 3 of Lemma \ref{genfun},
\[ \sum_{j \ odd} \frac{1-u^j}{1+u^j} \left( u^j (q-1)  + u^{2j}(2q-1) + 2q(u^{3j}+u^{4j} + \cdots) \right) \] simplifies to
\[ \frac{(q-1)u}{1-u^2} + \frac{u^2}{1-u^4}, \] and the result follows.
\end{proof}

As a corollary, we derive a generating function for $k(\AO^+) + k(\AO^-)$.

\begin{theorem} \label{sumgen} In odd characteristic,
\[ 1 + \sum_{n \geq 1} u^n [k(\AO^+(n,q)) + k(\AO^-(n,q))] \] is equal to
\[ \prod_i \frac{(1+u^{2i-1})^4}{1-qu^{2i}} \cdot \left( 1 + \frac{u^2+(q-1)u}{1-u^2} \right) .\] \end{theorem}

\begin{proof} By Lemma \ref{exact} and Theorem \ref{orthoformodd},
\[ k(\AO^+(n,q)) + k(\AO^-(n,q)) \] is equal to $T_1+T_2+T_3$ where
$T_1$ is $k(\OO^+(n,q)) + k(\OO^-(n,q))$, and $T_2, T_3$ are the following sums over conjugacy classes of $\OO^+(n,q)$ and $\OO^-(n,q)$:
\[ T_2 = \sum_C  \sum_{i \ even \atop a_i \neq 0} 1 \]
\[ T_3 = \sum_C \sum_{i \ odd \atop a_i \neq 0} f_i.\]

From \cite{W}, $T_1$ is the coefficient of $u^n$ in \[ \prod_i \frac{(1+u^{2i-1})^4}{1-qu^{2i}} \]

To compute the generating function for $T_2$, we take Wall's generating function for $T_1$, divide it by the generating function for unipotent conjugacy classes in part 1 of Lemma \ref{genfunO}, and multiply it by the generating function for the weighted sum over unipotent classes in part 2 of Lemma \ref{genfunO}. We conclude that $T_2$ is the coefficient of $u^n$ in
\[ \frac{u^4}{1-u^4} \prod_i \frac{(1+u^{2i-1})^4}{1-qu^{2i}}  \]

To compute the generating function for $T_3$, we take Wall's generating function for $T_1$, divide it by the generating function for unipotent conjugacy classes in part 1 of Lemma \ref{genfunO} and multiply it by the generating function for the weighted sum over unipotent classes in part 3 of Lemma \ref{genfunO}. We conclude that $T_3$ is the coefficient of $u^n$ in \[ \left( \frac{(q-1)u}{1-u^2} + \frac{u^2}{1-u^4} \right) \prod_i \frac{(1+u^{2i-1})^4}{1-qu^{2i}}. \]

Since \[ 1 + \frac{u^4}{1-u^4} + \frac{(q-1)u}{1-u^2} + \frac{u^2}{1-u^4} = 1 + \frac{u^2+(q-1)u}{1-u^2},\] the result follows. \end{proof}

Next, we derive a generating function for $k(\AO^+) - k(\AO^-)$.

\begin{theorem} \label{dirgen} In odd characteristic,
\[ 1 + \sum_{n \geq 1} u^n [k(\AO^+(n,q)) - k(\AO^-(n,q))] \] is equal to
\[ \frac{1}{1-u^2} \prod_i \frac{(1-u^{4i-2})}{1-qu^{4i}} .\] \end{theorem}

\begin{proof} By Lemma \ref{exact} and Theorem \ref{orthoformodd},
\[ k(\AO^+(n,q)) - k(\AO^-(n,q)) \] is equal to $T_1+T_2+T_3$, where
$T_1$ is $k(\OO^+(n,q)) - k(\OO^-(n,q))$ ,
\[ T_2 = \sum_{C^+} \sum_{i \ even \atop a_i \neq 0} 1  - \sum_{C^-} \sum_{i \ even \atop a_i \neq 0} 1 \]
\[ T_3 = \sum_{C^+} \sum_{i \ odd \atop a_i \neq 0} f_i - \sum_{C^-} \sum_{i \ odd \atop a_i \neq 0} f_i.\] Here $C^+$ ranges over conjugacy classes of $\OO^+(n,q)$, and $C^-$ ranges over conjugacy classes of $\OO^-(n,q)$.

From Wall \cite{W}, $T_1$ is the coefficient of $u^n$ in
\[ \prod_i \frac{1-u^{4i-2}}{1-qu^{4i}}.\]

To compute the generating function of $T_2$, we take the generating function for $T_1$, multiply it by $\prod_i (1-u^{4i})$ (which corresponds to removing the unipotent part). Then to add in the weighted unipotent part, one multiplies by
\[ \sum_{j \ even} (u^{2j}+u^{4j}+\cdots) \prod_{i \neq j \atop i \ even} (1+u^{2i}+u^{4i}+\cdots), \] which is equal to
\[ \frac{u^4}{1-u^4} \prod_i \frac{1}{1-u^{4i}}.\] We conclude that $T_2$ is the coefficient of $u^n$ in \[ \frac{u^4}{1-u^4} \prod_i \frac{(1-u^{4i-2})}{1-qu^{4i}}.\]

To compute the generating function of $T_3$, we take the generating function for $T_1$, multiply it by $\prod_i (1-u^{4i})$ (which corresponds to removing the unipotent part). Then to add in the weighted unipotent part, one multiplies by
\[ \sum_{j \ odd} u^{2j} \prod_{i \ even} (1+u^{2i}+u^{4i}+\cdots). \] Note that the terms involving $f_i$ canceled out (except for the $a_i=2$ case). The upshot is that the generating function for $T_3$ is \[ \frac{u^2}{1-u^4} \prod_i \frac{(1-u^{4i-2})}{1-qu^{4i}}.\]

Since \[ 1 + \frac{u^4}{1-u^4} + \frac{u^2}{1-u^4} = \frac{1}{1-u^2},\] the proof
is complete.
\end{proof}

\subsection{Character approach for $k(\AO)$, any characteristic} \label{repO}

Next we consider orthogonal groups.
In this case,  the natural module $V$ can be identified with its character group and the nontrivial $G$-orbits
correspond to nonzero vectors of $v$ of a given norm.

First consider the case $G=\OO^{\epsilon}(n,q)$ with $q$ odd.    The stabilizers are thus
$\AO^{\epsilon}(m-2,q)$ (for an isotropic vector) and $(q-1)/2$ copies each of $\OO^+(n-2,q)$ and $\OO^-(n-2,q)$. Note that we use the convention that $\OO^{\epsilon}(0,q)$
and $\AO^{\epsilon}(0,q)$ are empty if $\epsilon = -$ and are the trivial group if $\epsilon = +$. Similarly, $\AO^{\epsilon}(-1,q)$ is the empty set.
And as in earlier cases, the trivial group has one conjugacy class and the empty set has zero conjugacy classes. This yields the following result.

\begin{lemma}  \label{l:aoodd}  Let $q$ be odd and $n \ge 1$.  Then
$k(\AO^{\epsilon}(n,q)) =
k(\OO^{\epsilon}(n,q)) + k(\AO^{\epsilon}(n-2,q)) + (q-1)(k(\OO^+(n-1,q)) + k(\OO^-(n-1,q)))/2.$
\end{lemma}

As a corollary, we obtain a second proof of Theorems \ref{sumgen} and \ref{dirgen}.

\begin{proof} (Second proof of Theorem \ref{sumgen}) Define
\[ K_O(u) = 1 + \sum_{n \geq 1} u^n [k(\OO^+(n,q)) + k(\OO^-(n,q))] \]
\[ A_O(u) = 1 + \sum_{n \geq 1} u^n [k(\AO^+(n,q)) + k(\AO^-(n,q))].\]
By the above recursion, we have that for all $n$,
\begin{eqnarray*}
k(\AO^+(n,q)) & = & k(\OO^+(n,q)) + k(\AO^+(n-2,q))\\
& &  + \frac{q-1}{2} [k(\OO^+(n-1,q)) +
k(\OO^-(n-1,q))].
\end{eqnarray*}
\begin{eqnarray*}
k(\AO^-(n,q)) & = & k(\OO^-(n,q)) + k(\AO^-(n-2,q))\\
& &  + \frac{q-1}{2} [k(\OO^+(n-1,q)) + k(\OO^-(n-1,q))].
\end{eqnarray*}
Adding these two equations gives
\begin{eqnarray*}
& & k(\AO^+(n,q)) + k(\AO^-(n,q)) \\
& = & k(\OO^+(n,q)) + k(\OO^-(n,q)) + k(\AO^+(n-2,q)) + k(\AO^-(n-2,q)) \\
& & + (q-1)[k(\OO^+(n-1,q)) + k(\OO^-(n-1,q))].
\end{eqnarray*} Multiplying this by $u^n$ and summing over $n \geq 0$ gives
that \[ A_O(u) = K_O(u) + u^2 A_O(u) + u(q-1) K_O(u).\] Thus
\[ A_O(u) = \frac{K_O(u)}{1-u^2} \left(1 + u(q-1) \right).\] The result now follows from Wall's formula
\[ K_O(u) = \prod_i \frac{(1+u^{2i-1})^4}{1-qu^{2i}}.\]
\end{proof}

\begin{proof} (Second proof of Theorem \ref{dirgen}) Let
\[ D(u) = 1 + \sum_{n \geq 1} u^n [k(\OO^+(n,q)) - k(\OO^-(n,q))] \]
\[ B(u) = 1 + \sum_{n \geq 1} u^n [k(\AO^+(n,q)) - k(\AO^-(n,q))] \]

From Lemmas \ref{l:aoodd}, we have that
\begin{eqnarray*}
& & k(\AO^+(n,q)) - k(\AO^-(n,q)) \\
& = & k(\OO^+(n,q)) - k(\OO^-(n,q)) \\
& & + k(\AO^+(n-2,q)) - k(\AO^-(n-2,q).
\end{eqnarray*}

Multiplying this equation by $u^n$ and summing over all $n \geq 0$ gives
that \[ B(u) = D(u) + u^2 B(u).\] Thus $B(u) = D(u)/(1-u^2)$, and the result follows from Wall's formula \[ D(u) = \prod_i \frac{(1-u^{4i-2})}{1-qu^{4i}}.\]
\end{proof}

Finally we turn to characteristic 2. In this case odd dimensional orthogonal groups are isomorphic to symplectic groups, so we need only consider the even dimensional case. So consider $G = \OO^{\epsilon}(n,q)$ with $q$ and $n$ both even.  The argument is similar.  The only
difference is that the stabilizer of a vector of nonzero norm in $\AO^{\epsilon}(n,q)$ is $\Sp(n-2,q) \times \Z/2$ and so:

\begin{lemma} \label{l:asoeven}   Let $q$ be even and $n \ge 2$ be even.
 Then $k(\AO^{\epsilon}(n,q))$ is equal to $k(\OO^{\epsilon}(n,q)) + k(\AO^{\epsilon}(n-2,q)) + 2(q-1)(k(\Sp(n-2,q))$.
 \end{lemma}

For $n=2$ we used the convention that $k(\AO^+(0,q))=1$ and $k(\Sp(0,q))=1$, and that $k(\AO^-(0,q))=0$.

Next using Lemma \ref{l:asoeven} (and generating functions for $k(\Sp)$ and $k(\OO)$) we derive generating functions for $k(\AO^{\pm}(2n,q))$ in even characteristic.

\begin{theorem} \label{sumgeneven} Let $q$ be even. Then $k(\AO^+(2n,q)) + k(\AO^-(2n,q))$ is the coefficient of $u^n$ in \[ \frac{1}{1-u} \left( K_O(u) + 4(q-1)u K_{Sp}(u) \right) \] where
\[ K_O(u) = \prod_{i \geq 1} \frac{(1+u^i)(1+u^{2i-1})^2}{1-qu^i} \]
\[ K_{Sp}(u) = \prod_{i \geq 1} \frac{(1-u^{4i})}{(1-u^{4i-2})(1-u^i)(1-qu^i)}.\]
\end{theorem}

\begin{proof} Define generating functions,
\[ K_{Sp}(u) = 1 + \sum_{n \geq 1} k(\Sp(2n,q)) u^n \]
\[ K_O(u) = 1 + \sum_{n \geq 1} [k(\OO^+(2n,q)) + k(\OO^-(2n,q))] u^n \]
\[ A_O(u) = 1 + \sum_{n \geq 1} [k(\AO^+(2n,q)) + k(\AO^-(2n,q))] u^n. \]

Now take the recursions for $k(\AO^+(2n,q))$ and $k(\AO^-(2n,q))$ in Lemma \ref{l:asoeven}, multiply them by $u^n$ and sum over all $n \geq 0$. We
conclude that
\[ A_O(u) = K_O(u) + u A_O(u) + 4(q-1)u K_{Sp}(u).\]
Thus \[ A_O(u) = \frac{1}{1-u} \left( K_O(u) + 4(q-1)u K_{Sp}(u) \right).\]

From \cite{FG1},
\[ K_O(u) = \prod_i \frac{(1+u^i)(1+u^{2i-1})^2}{(1-qu^i)} \]
\[ K_{Sp}(u) = \prod_i \frac{(1-u^{4i})}{(1-u^{4i-2})(1-u^i)(1-qu^i)}.\]
\end{proof}

\begin{theorem} \label{difgeneven} Let $q$ be even. Then $k(\AO^+(2n,q)) - k(\AO^-(2n,q))$ is the coefficient of $u^n$ in
\[ \frac{1}{1-u} \prod_{i \geq 1} \frac{1-u^{2i-1}}{1-qu^{2i}}.\]
\end{theorem}

\begin{proof} Define generating functions
\[ D(u) = 1 + \sum_{n \geq 1} u^n [k(\OO^+(2n,q)) - k(\OO^-(2n,q))] \]
\[ B(u) = 1 + \sum_{n \geq 1} u^n [k(\AO^+(2n,q)) - k(\AO^-(2n,q))] \]

Multiply the recursions for $k(\AO^+(2n,q))$ and $k(\AO^-(2n,q))$ in Lemma \ref{l:asoeven} by $u^n$, sum over all $n \geq 0$, and subtract to obtain
\[ B(u) = D(u) + u B(u).\]

Using Wall's formula \cite{W} for $D(u)$, we conclude that \[ B(u) = \frac{1}{1-u} D(u) = \frac{1}{1-u} \prod_i \frac{1-u^{2i-1}}{1 - qu^{2i}}.\]
\end{proof}

\subsection{Bounds on $k(\AO)$} \label{Obounds}

This section derives bounds on $k(\AO)$.

We begin with the case of odd characteristic and even dimension.

\begin{cor} Let $q$ be odd. Then $k(\AO^{\pm}(2n,q)) \leq 29 q^n$.
\end{cor}

\begin{proof} From Theorem \ref{sumgen},
\[ k(\AO^+(2n,q)) + k(\AO^-(2n,q)) \] is the coefficient of $u^{2n}$ in
\[ \prod_i \frac{(1+u^{2i-1})^4}{1-qu^{2i}} \cdot \left( 1 + \frac{u^2+(q-1)u}{1-u^2} \right).\] Rewrite this as \[ \prod_i \frac{1-u^{2i}}{1-qu^{2i}} \prod_i \frac{(1+u^{2i-1})^4}{1-u^{2i}} \cdot \left( 1 + \frac{u^2+(q-1)u}{1-u^2} \right).\]

As in the symplectic case, the coefficient of $u^{2n-2m}$ in $\prod_i \frac{1-u^{2i}}{1-qu^{2i}}$ is at most $q^{n-m}$. Thus \[ k(\AO^+(2n,q)) + k(\AO^-(2n,q)) \] is at most
\[ q^n \sum_{m \geq 0} \frac{1}{q^{m}} \rm{Coef. \ } u^{2m} \rm{\ in} \prod_i \frac{(1+u^{2i-1})^4}{1-u^{2i}} \cdot \left( 1 + \frac{u^2+(q-1)u}{1-u^2} \right) \]
which is equal to $q^n/2$ multiplied by
\begin{eqnarray*}
& & \prod_i \frac{(1+u^{2i-1})^4}{1-u^{2i}} \cdot \left( 1 + \frac{u^2+(q-1)u}{1-u^2} \right) \\
& & + \prod_i \frac{(1-u^{2i-1})^4}{1-u^{2i}} \cdot \left( 1 + \frac{u^2-(q-1)u}{1-u^2} \right)
\end{eqnarray*} evaluated at $u=1/\sqrt{q}$. Since $q \geq 3$, we conclude that \[ k(\AO^+(2n,q)) + k(\AO^-(2n,q)) \leq 53 q^n.\]

From Theorem \ref{dirgen}, \[ k(\AO^+(2n,q)) - k(\AO^-(2n,q)) \] is the coefficient of $u^n$ in \[ \frac{1}{1-u} \prod_i \frac{1-u^{2i-1}}{1-qu^{2i}}.\] This is analytic for $|u| < \frac{1}{q} + \epsilon$, so Lemmas \ref{bit} and \ref{pentag} imply an upper bound of \[ q^n \frac{1}{1-1/q} \prod_i \frac{1+1/q^{2i-1}}{1-1/q^{2i-1}} \leq 3.3 q^n.\]

Combining the results of the previous two paragraphs proves the corollary, as $(53+3.3)/2 \leq 29$.
\end{proof}

\begin{cor} Let $q$ be odd. Then $k(\AO^{\pm}(2n,q)) \leq q^{2n}$.
\end{cor}

\begin{proof} The result follows from the previous corollary whenever $29 q^n \leq q^{2n}$. So we only need to check the cases $n=1$, or $n=2,q=3,5$, or $n=3,q=3$. These cases are easily checked from our generating function for $k(\AO^{\pm}(2n,q))$.
\end{proof}

Next we treat the case of odd dimensional groups in odd characteristic.   In this case, the upper bound is not of the form
constant times $q^{\rm rank}$.   This is because every element in the classical group has eigenvalue $1$.
\begin{cor} Let $q$ be odd. Then $k(\AO(2n+1,q)) \leq 20 q^{n+1}$.
\end{cor}

\begin{proof} We prove this by induction on $n$. By our earlier recursion,
\begin{eqnarray*}
k(\AO(2n+1,q)) & = & k(\OO(2n+1,q)) + k(\AO(2n-1,q)) \\
& & + \frac{(q-1)}{2} [k(\OO^+(2n,q))+k(\OO^-(2n,q))].
\end{eqnarray*}

By \cite{FG1},
\[ k(\OO(2n+1,q)) \leq 14.2 q^n \] and
\[ k(\OO^+(2n,q))+k(\OO^-(2n,q)) \leq 16.3 q^n. \]
Thus \[ k(\AO(2n+1,q)) \leq k(\AO(2n-1,q)) + 14.2 q^n + 8.2 q^{n+1}.\]
By induction, $k(\AO(2n-1,q)) \leq 20 q^n$, so the result follows since
\[ 20 q^n + 14.2 q^n + 8.2 q^{n+1} \leq 20 q^{n+1} \] for $q \geq 3$.
\end{proof}

\begin{cor} Let $q$ be odd. Then $k(\AO(2n+1,q)) \leq q^{2n+1}$.
\end{cor}

\begin{proof} By the previous corollary, the result holds if $20 \leq q^n$. So we need only check the cases $n=0$, $n=1$, or $n=2,q=3$. The generating function (Theorem \ref{sumgen}) implies that $k(\AO(1,q)) = (q+3)/2$ and $k(\AO(3,q)) = (q^2+10q+5)/2$, and shows that the exact value of $k(\AO(5,3))$ is less than $243$.
\end{proof}

Next we turn to the case of even characteristic.

\begin{cor} Let $q$ be even. Then $k(\AO^{\pm}(2n,q)) \leq 60 q^n$.
\end{cor}

\begin{proof} From Theorem \ref{sumgeneven}, $k(\AO^+(2n,q)) + k(\AO^-(2n,q))$ is equal to the coefficient of $u^n$ in
\begin{eqnarray*}
& & \prod_i \frac{1-u^i}{1-qu^i} \\
& & \cdot \frac{1}{1-u} \left[ \prod_i \frac{(1+u^i)(1+u^{2i-1})^2}{(1-u^i)} +
4(q-1)u \prod_i \frac{(1-u^{4i})}{(1-u^{4i-2})(1-u^i)^2} \right].
\end{eqnarray*}

Arguing as for the symplectic groups, this is at most $q^n$ multiplied by
\[ \frac{1}{1-\frac{1}{q}} \left[ \prod_i \frac{(1+1/q^i)(1+1/q^{2i-1})^2}{(1-1/q^i)} +
\frac{4(q-1)}{q} \prod_i \frac{(1-1/q^{4i})}{(1-1/q^{4i-2})(1-1/q^i)^2} \right] \] which is at most $111.6 q^n$ since $q \geq 2$.

From Theorem \ref{difgeneven}, $k(\AO^+(2n,q)) - k(\AO^-(2n,q))$ is equal to the coefficient of $u^n$ in \[ \frac{1}{1-u} \prod_i \frac{1-u^{2i-1}}{1-qu^{2i}}.\]
Since this is analytic for $|u|<q^{-1}+ \epsilon$, Lemma \ref{bit} gives that
$k(\AO^+(2n,q)) - k(\AO^-(2n,q))$ is at most
\[ q^n \frac{1}{1-1/q} \prod_i \frac{1+1/q^{2i-1}}{1-1/q^{2i-1}} \leq 8.4 q^n. \]

The corollary now follows since $(111.6 + 8.4)/2 = 60$.
\end{proof}

\begin{cor} Let $q$ be even. Then $k(\AO^{\pm}(2n,q)) \leq q^{2n}$ except for:
$k(\AO^+(2,2))=5$, $k(\AO^-(2,2))=5$, $k(\AO^+(4,2))=20$, $k(\AO^-(4,2))=18$, and $k(\AO^-(6,2))=65$. \end{cor}

\begin{proof} By the previous corollary, $k(\AO^{\pm}(2n,q)) \leq q^{2n}$ if $60 \leq q^n$. So we need only check the cases $n=1$ or $q=2, 2 \leq n \leq 5$, or $q=4,n=2$.
So the only infinite family of cases to check is when $n=1$, in which case the generating function gives $k(\AO^{\pm}(2,q))=5q/2$. The remaining finite number of cases can be checked immediately from the generating function.
\end{proof}

\end{document}